\documentclass{birkjour}

\usepackage{amsfonts, amsthm, amssymb, amsmath, mathrsfs}
\usepackage[all]{xy}

\ifpdf
	\usepackage[unicode,pdftex]{hyperref}
	\input glyphtounicode
\else
	\usepackage[unicode,dvipdfm]{hyperref}
\fi

\newtheorem{thm}{Theorem}[section]
\newtheorem{prop}[thm]{Proposition}
\newtheorem{lem}[thm]{Lemma}
\newtheorem{cor}[thm]{Corollary}

\theoremstyle{definition}

\newtheorem{rmk}[thm]{Remark}

\numberwithin{equation}{section}

\newcommand{\PP}{\ensuremath{\mathbb{P}}}

\begin{document}

\title[Irreducibility of the Hilbert scheme of curves in $\Bbb P^4$ of degree $g+2$]
 {Irreducibility of  the Hilbert scheme of smooth curves in $\Bbb P^4$ of degree $g+2$ and genus $g$}

\author[Changho Keem]{Changho Keem}

\address{
Department of Mathematics,
Seoul National University\\
Seoul 151-742,  
South Korea}

\email{ckeem1@gmail.com}

\thanks{Both authors were not  supported by National Research Foundation}

\author[Yun-Hwan Kim]{Yun-Hwan Kim}
\address{Department of Mathematics,
Seoul National University\\
Seoul 151-742, 
South Korea}
\email{yunttang@snu.ac.kr}

\subjclass{Primary 14H10; Secondary 14C05}

\keywords{Hilbert scheme, algebraic curves, linear series}

\date{\today}

\begin{abstract}
We denote by $\mathcal{H}_{d,g,r}$ the Hilbert scheme of smooth curves, which is the union of components whose general point corresponds to a smooth irreducible and non-degenerate curve of degree $d$ and genus $g$ in $\PP^r$. In this note, we show that any non-empty $\mathcal{H}_{g+2,g,4}$ is irreducible without any restriction on the genus $g$. 
Our result augments the irreducibility result obtained earlier by Hristo Iliev(2006), in which several low genus $g\le 10$ cases have been left untreated.
\end{abstract}

\maketitle

\section{\quad An overview, preliminaries and basic set-up}

Given non-negative integers $d$, $g$ and $r\ge 3$, let $\mathcal{H}_{d,g,r}$ be the Hilbert scheme of smooth curves parametrizing smooth irreducible and non-degenerate curves of degree $d$ and genus $g$ in $\PP^r$.

After Severi asserted that $\mathcal{H}_{d,g,r}$ is irreducible for $d\ge g+r$ in \cite[Anhang G, p. 368]{Sev} with an incomplete proof, the irreducibility of $\mathcal{H}_{d,g,r}$ has been studied by several authors. Most noteworthy result regarding the irreducibility of $\mathcal{H}_{d,g,r}$ is due to L. Ein,  who proved Severi's claim for  \& $r=4$ (\& for $r=3$ as well); cf. \cite[Theorem 7]{E2} \&  \cite[Theorem 4]{E1}, see also \cite[Theorem 1.5]{KK}, \cite[Proposition 2.1 \& Proposition 3.2]{KKL},  \cite[Theorem 3.1]{I}, \cite[Theorem 2.1]{KK1}
for a different proof and several extensions concerning the irreducibility of $\mathcal{H}_{d,g,3}$ in the range $d\ge g$ . 

For  smooth curves in $\mathbb{P}^4$, the irreducibility of $\mathcal{H}_{g+3,g,4}$ ($g\ge 5$) and  $\mathcal{H}_{g+2,g,4}$ ($g\ge 11$) has been pushed forward by Hristo Iliev \cite[Theorem 3.2]{I} 
beyond the range $d\ge g+4$ which has been known by a work of L. Ein \cite[Theorem 7]{E2}. However, the irreducibility of $\mathcal{H}_{g+2,g,4}$ (and the non-emptyness as well)  for lower genus $g\le 10$ has been left unsettled.  In this article, we show that $\mathcal{H}_{g+2,g,4}$ is non-empty and irreducible for $7\le g\le 10$ which  in turn implies (together with  \cite[Theorem 3.2(b)]{I}) that any non-empty $\mathcal{H}_{g+2,g,4}$ is irreducible  without any restriction on the genus $g$; cf. Theorem \ref{H_g+2,g,4}, Corollary \ref{cor} and Remark \ref{finalremark}.

\vskip 6pt
Before proceeding, we recall several related results which are rather well-known; cf. \cite{AC1} and \cite{AC2}.
Let $\mathcal{M}_g$ be the moduli space of smooth curves of genus $g$. For any given isomorphism class $[C] \in \mathcal{M}_g$ corresponding to a smooth irreducible curve $C$, there exist a neighborhood $U\subset \mathcal{M}_g$ of the class $[C]$ and a smooth connected variety $\mathcal{M}$ which is a finite ramified covering $h:\mathcal{M} \to U$, as well as  varieties $\mathcal{C}$, $\mathcal{W}^r_d$ and $\mathcal{G}^r_d$ proper over $\mathcal{M}$ with the following properties:
\begin{enumerate}
\item[(1)] $\xi:\mathcal{C}\to\mathcal{M}$ is a universal curve, i.e. for every $p\in \mathcal{M}$, $\xi^{-1}(p)$ is a smooth curve of genus $g$ whose isomorphism class is $h(p)$,
\item[(2)] $\mathcal{W}^r_d$ parametrizes the pairs $(p,L)$ where $L$ is a line bundle of degree $d$ and $h^0(L) \ge r+1$ on $\xi^{-1}(p)$,
\item[(3)] $\mathcal{G}^r_d$ parametrizes the couples $(p, \mathcal{D})$, where $\mathcal{D}$ is possibly an incomplete linear series of degree $d$ and dimension $r$ on $\xi^{-1}(p)$ - which is usually denoted by $g^r_d$. 
\end{enumerate}

\vskip 6pt
Let $\widetilde{\mathcal{G}}$ be the union of components of $\mathcal{G}^{r}_{d}$ whose general element $(p,\mathcal{D})$ corresponds to a very ample linear series $\mathcal{D}$ on the curve $C=\xi^{-1}(p)$. Note that an open subset of $\mathcal{H}_{d,g,r}$ consisting of points corresponding to smooth irreducible and non-degenerate curves is a $\mathbb{P}\textrm{GL}(r+1)$-bundle over an open subset of $\widetilde{\mathcal{G}}$. Hence the irreducibility of $\widetilde{\mathcal{G}}$ guarantees the irreducibility of $\mathcal{H}_{d,g,r}$. 
We also make a note of the following well-known facts regarding the scheme $\mathcal{G}^{r}_{d}$; cf.  \cite{AC1}, \cite[Chapt. 21, \S 3, 5, 6, 11,12]{ACGH2} and \cite[\S 2.a, p. 67]{H1}.
\begin{prop}\label{facts}
For non-negative integers $d$, $g$ and $r$, let $\rho(d,g,r):=g-(r+1)(g-d+r)$ be the Brill-Noether number.
	\begin{enumerate}
	\item[\rm{(1)}] The dimension of any component of $\mathcal{G}^{r}_{d}$ is at least $3g-3+\rho(d,g,r)$ which is denoted by $\lambda(d,g,r)$.
\item[\rm(2)] 
$\mathcal{G}^{1}_{d}$ is smooth and irreducible of dimension $\lambda(d,g,1)$ if $g>1, d\ge 2$ and $d\le g+1$.
	\end{enumerate}
\end{prop}

\vskip 4pt
We will utilize the  following upper bound of the dimension of a certain  irreducible component of $\mathcal{W}^r_d$, which was proved  and
used effectively in \cite{I}.

\begin{prop}[\rm{\cite[Proposition 2.1]{I}}]\label{wrdbd}
Let $d,g$ and $r\ge 2$ be positive integers such that  $d\le g+r-2$ and let $\mathcal{W}$ be an irreducible component of $\mathcal{W}^{r}_{d}$. For a general elment $(p,L)\in \mathcal{W}$, let $b$ be the degree of the base locus of the line bundle $L=|D|$ on $C=\xi^{-1}(p)$. Assume further that for a general $(p,L)\in \mathcal{W}$ the curve $C=\xi^{-1}(p)$ is not hyperelliptic. If the moving part of $L=|D|$ is very ample and $r\ge3$, then
$$\dim \mathcal{W}\le 3d+g+1-5r-2b.$$
\end{prop}

\vskip 4pt

For notations and conventions, we usually follow those in \cite{ACGH} and \cite{ACGH2}; e.g. $\pi (d,r)$ is the maximal possible arithmetic genus of an irreducible and non-degenerate curve of degree $d$ in $\PP^r$;
$
\pi(d,r):=\binom{m}{2}(r-1)+m\epsilon
$
where $m=\left[\frac{d-1}{r-1}\right]$ and $d-1=m(r-1)+\epsilon$. 
Throughout we work over the field of complex numbers.
\section{\quad Irreducibility of $\mathcal{H}_{g+2,g,4}$}
We first recall that the  irreducibility of  
$\mathcal{H}_{g+3,g,4}$ for $g\ge 5$ has been shown by Hristo Iliev;  \cite[Theorem 3.2(a)]{I}. We also remark that this result of Hristo Iliev has been stated with the fullest possible generality; note that  $\pi(g+3,4)<g$ for $g\le 4$ and hence $\mathcal{H}_{g+3,g,4}=\emptyset $ if $g\le 4$.

In the same vein, one may easily see that $\mathcal{H}_{g+2,g,4}=\emptyset$ for $g\le 6$; by the Castelnuovo genus bound, one checks that there is no smooth and non-degenerate curve in $\PP^4$ of degree $g+2$ and genus $g$ if $g\le 6$. 
Therefore in conjunction with the theorem of Hristo Iliev \cite[Theorem 3.2(b)]{I}, i.e.  $\mathcal{H}_{g+2,g,4}$ being irreducible for $g\ge 11$, we shall assume $7\le g\le 10$ for the rest of this section. 
 The main result of this article is the following theorem.
 
\begin{thm}\label{H_g+2,g,4}
$\mathcal{H}_{g+2,g,4}$ is irreducible for any $g$ with $7\le g\le 10$.
\end{thm}

Consequently Theorem \ref{H_g+2,g,4} together with a result of Hristo Iliev \cite[Theorem 3.2(b)]{I} readily imply the following statement.
\begin{cor}\label{cor} Any non-empty $\mathcal{H}_{g+2,g,4}$ is irreducible.
 \end{cor}

\begin{rmk}
{\rm It is worthwhile to note that the genus range $g\ge 11$ in \cite[Theorem 3.2(b)]{I} is exactly the range where the Brill-Noether number $\rho (g+2,g,4)$ is strictly positive so that there exists a unique component of the Hilbert scheme  $\mathcal{H}_{g+2,g,4}$ dominating the moduli space $\mathcal{M}_g$.
For curves of genus $g$ in the range $7\le g\le 10$ -- in which case $\rho (g+2,g,4)\le 0$ -- we will see in Remark \ref{finalremark} that $\mathcal{H}_{g+2,g,4}\neq\emptyset$ and the unique component of $\mathcal{H}_{g+2,g,4}$ is indeed the component which dominates the irreducible locus $\mathcal{M}^1_{g, g-4}$ in $\mathcal{M}_g$  consisting of $(g-4)$-gonal curves.}

\end{rmk}

The following lemma,
which is an intermediate step toward the proof of the irreducibility of $\mathcal{H}_{g+2,g,4}$, asserts that a general element in any component of $\mathcal{H}_{g+2,g,4}$ corresponds to a smooth curve in $\mathbb{P}^4$ which is linearly normal.

\begin{lem}\label{dualbirva}
Let ${\mathcal{G}}\subset \mathcal{G}^{4}_{g+2}$ be an irreducible component whose general element $(p, \mathcal{D})$ is a very ample linear series $\mathcal{D}$ on the curve $C=\xi^{-1}(p)$ and assume $7\le g\le 10$. Then 
\begin{enumerate}
\item[\rm{(1)}] $\mathcal{D}$ is complete and $\dim{\mathcal{G}}=4g-13$.
\item[\rm{(2)}] a general element of the component ${\mathcal{W}}^{\vee}\subset \mathcal{W}^1_{g-4}$ consisting of the residual  series (with respect to the canonical series on the corresponding curve) of those elements in  ${\mathcal{G}}$ is a complete pencil.
\end{enumerate}
\end{lem}

\begin{proof}
\quad By Proposition \ref{facts}, we have
\[
\lambda (g+2,g,4)=3g-3+\rho(g+2,g,4)=4g-13\le \dim {\mathcal{G}}.
\]
We set $r:=h^0(C, |\mathcal{D}|)-1$ for a general $(p,\mathcal{D})\in{\mathcal{G}}$.
Let $\mathcal{W}\subset \mathcal{W}^{r}_{g+2}$ be the component containing the image of the natural rational map 
${\mathcal{G}}\overset{\iota}{\dashrightarrow} \mathcal{W}^{r}_{g+2}$ with $\iota (\mathcal{D})=|\mathcal{D}|$.
Since $\dim{\mathcal{G}}\le\dim \mathcal{W} +\dim {\mathbb{G}}(4,r)$, it follows  by Proposition  \ref{wrdbd} that 
\[
\lambda (g+2,g,4)=4g-13\le \dim \mathcal{G}\le (4g+7-5r)+5(r-4)=4g-13,
\]
 hence 
\begin{equation}\label{dimension} 
\dim \mathcal{G}=4g-13 ~~~~~~\hskip 12pt \text{and} ~~~~ \hskip 12pt\dim \mathcal{W}=4g+7-5r.
\end{equation}
\vskip 6pt
Let $\mathcal{W}^\vee\subset \mathcal{W}^{r-3}_{g-4}$ be the locus consisting  of the residual  series (with respect to the canonical series on the corresponding curve) of those elements in $\mathcal{W}$, i.e. $\mathcal{W}^\vee =\{(p, \omega_C\otimes L^{-1}): (p, L)\in\mathcal{W}\}.$

Assume that $r\ge 5$ and we will argue by contradiction for each $g$ with 
$7\le g\le 10$. 

\begin{enumerate}
\item[(a)] $g=7$: We have $g^2_3$ on $C=\xi^{-1}(p)$ corresponding to an element in $\mathcal{W}^\vee \subset\mathcal{W}^{r-3}_{g-4}$, contradicting Clifford's theorem. 
\item[(b)] $g=8$: We have  $g^2_4$ on $C=\xi^{-1}(p)$ corresponding to an element in $\mathcal{W}^\vee\subset\mathcal{W}^{r-3}_{g-4}$, hence $C$ is hyperelliptic by Clifford's theorem. However a hyperelliptic curve of genus $g$ cannot have a very ample $g^4_{g+2}$. 
\item[(c)] $g=9$: We have $g^2_5$ on $C=\xi^{-1}(p)$ and since $C$ is not a plane quintic, our $g^2_5$ has a base point and hence $C$ is a hyperelliptic curve, which  cannot have a very ample $g^4_{g+2}$. 
\item[(d)]  $g=10$: In this case, we have either a  $g^2_6$ (when $r=5$) or a $g^3_6$ (when $r=6$) on $C=\xi^{-1}(p)$ corresponding to an element in $\mathcal{W}^\vee\subset\mathcal{W}^{r-3}_{g-4}$. If there were a $g^3_6$, $C$ is an hyperelliptic curve on which there does not exist  a very ample $g^4_{g+2}$. Therefore $r=5$. Note that our $g^2_6$ is  base-point-free, for otherwise the same reasoning as in (c) applies. Therefore it follows that  $C$ is either trigonal with a unique trigonal pencil $g^1_3$ so that $2g^1_3=g^2_6$ or a smooth plane sextic with $|K_C-g^2_6|=|2g^2_6|=g^5_{12}$, where $g^2_6$ is the unique very ample net of degree $6$. Suppose that $C$ is a smooth plane sextic. We recall that  two smooth plane curves of the same degree $d\ge 4$ are isomorphic if and only if they are projectively equivalent. Hence the family of smooth plane curves of degree $d=6$ moves in 
$\dim\mathbb{P}H^0(\mathbb{P}^2, \mathcal{O}(d))-\dim\mathbb{P}\text{GL}(3)=\frac{d(d+3)}{2}-\dim\mathbb{P}\text{GL}(3)=27-8=19$
dimensional  locus $\widetilde{\mathcal{M}}$ in $\mathcal{M}_{g}.$  
In the sequence of natural rational maps 
$\mathcal{G}\overset{\iota}\dashrightarrow
\mathcal{W}\overset{\zeta}\dashrightarrow\mathcal{W}^\vee\overset{\eta}\dashrightarrow\widetilde{\mathcal{M}}\subset\mathcal{M}_{g}$, we note that the rational map $\zeta$ -- which takes a complete linear series to its residual series -- is clearly birational and the projection map $\eta$ is also birational by the  uniqueness of $g^2_6$. Therefore we have  $\dim\mathcal{W}=\dim\mathcal{W}^\vee=\dim\widetilde{\mathcal{M}}=19$, contradicting (\ref{dimension}). If $C$ is trigonal with $g^2_6=2g^1_3$, we also have a sequence of rational maps 
$\mathcal{G}\overset{\iota}\dashrightarrow
\mathcal{W}\overset{\zeta}\dashrightarrow\mathcal{W}^\vee\overset{\eta}\dashrightarrow\mathcal{M}^1_{g,3}\subset\mathcal{M}_{g}$, where $\mathcal{M}^1_{g,3}$ is the irreducible locus of trigonal curves of genus $g$. 
Again the projection map $\eta$ is birational since there exists a unique trigonal pencil on any trigonal curve of genus $g$ when $g\ge 5$.  Hence $\dim\mathcal{W}=\dim\mathcal{W}^\vee=\dim{\mathcal{M}^1_{g,3}}=2g+1=21$, again contradicting (\ref{dimension}).
\end{enumerate}

\vskip 6pt
\noindent
Therefore it finally follows that $r=4$ and by (\ref{dimension}), we have
\begin{equation}\label{equal}\dim\mathcal{G}=\dim\mathcal{W}=\dim\mathcal{W}^{\vee}=4g-13.
\end{equation}
The second statement (2) is obvious from (1). 
\end{proof}

\vskip 6pt
The irreducibility of $\mathcal{H}_{g+2,g,4}$ follows easily as an
immediate consequence of Lemma \ref{dualbirva} together with Proposition \ref{facts}(2).
\vskip 6pt
\noindent
{\it Proof of Theorem \ref{H_g+2,g,4}.}
Retaining the same notations as before, let $\widetilde{\mathcal{G}}$ be the union of irreducible components $\mathcal{G}$ of $\mathcal{G}^{4}_{g+2}$ whose general element corresponds to a pair $(p,\mathcal{D})$ such that $\mathcal{D}$ is very ample linear series on $C:=\xi^{-1}(p)$. Let $\widetilde{\mathcal{W}}^\vee$ be the union of the  components $\mathcal{W}^\vee$ of $\mathcal{W}^1_{g-4}$, where $\mathcal{W}^\vee$ consists of the residual series of elements in a  component $\mathcal{G}$ of $\tilde{\mathcal{G}}$.
By Lemma \ref{dualbirva} (or (\ref{equal})), \begin{equation}\label{dominant}\dim \mathcal{W}^{\vee}=\dim\mathcal{G}=4g-13=\lambda (g-4,g,1)=\dim \mathcal{G}^1_{g-4}.\end{equation}
Since a general element of any component $\mathcal{W}^{\vee}\subset\widetilde{\mathcal{W}}^\vee\subset\mathcal{W}^1_{g-4}$ is 
a complete pencil
by Lemma \ref{dualbirva}, there is a natural rational map $\widetilde{\mathcal{W}}^\vee\overset{\kappa}{\dashrightarrow}\mathcal{G}^1_{g-4}$ with $\kappa(|\mathcal{D}|)=\mathcal{D}$   which is clearly injective on an open subset  $\widetilde{\mathcal{W}}^{\vee o}$ of $\widetilde{\mathcal{W}}^\vee$ consisting of those which are
complete pencils. Therefore the  rational map $\kappa$ is dominant by (\ref{dominant}). 
We also note that there is another natural rational map 
$\mathcal{G}^1_{g-4} \overset{\iota}{\dashrightarrow}\widetilde{\mathcal{W}}{^\vee}$ with $\iota (\mathcal{D})=|\mathcal{D|}$, which is an inverse to $\kappa$ (wherever it is defined).
Therefore it follows that $\widetilde{\mathcal{W}}^{\vee}$ is birationally equivalent to  the irreducible locus $\mathcal{G}^1_{g-4}$, hence $\widetilde{\mathcal{W}}^{\vee}$ is irreducible and so is $\widetilde{\mathcal{G}}$. Since $\mathcal{H}_{g+2,g,4}$ is a $\mathbb{P}\textrm{GL}(5)$-bundle over an open subset of $\widetilde{\mathcal{G}}$,  $\mathcal{H}_{g+2,g,4}$ is irreducible.
\qed

\vskip 6pt
\begin{rmk} \label{finalremark} {\rm We finally remark that $\mathcal{H}_{g+2,g,4}\neq\emptyset$
for $7\le g\le 10$. As was suggested by the proof of Theorem \ref{H_g+2,g,4}  one may argue  that $\mathcal{H}_{g+2,g,4}$ dominates (and is a $\mathbb{P}\textrm{GL}(5)$)-bundle over) the irreducible locus $\mathcal{M}^1_{g, g-4}$ consisting of $(g-4)$-gonal curves  as follows. Recall that the Clifford index $e$ of  a general $(e+2)$-gonal curve of genus $g\ge 2e+2$ can only be computed by the unique pencil computing the gonality as long as $e\neq 0$, i.e. there does not exist a $g^r_{2r+e}$ with $2r+e\le g-1$, $r\ge 2$; cf. \cite[Theorem]{B} or \cite[Corollary 1]{KK2}. Therefore on a general $(g-4)$-gonal curve $C$, the residual series of the unique $g^1_{g-4}$ is a very ample $g^4_{g+2}$; for otherwise there exists a $g^2_{g-2}=g^1_{g-4}\otimes\mathcal{O}_C(p+q)$ for some $p,q\in C$, computing the Clifford index of a general $(g-4)$-gonal curve contradicting the result just mentioned.  }
\end{rmk}
\vskip 12pt


\bibliographystyle{spmpsci} 

\end{document}